\documentclass{amsart}

\usepackage{enumerate}
\usepackage{amssymb}
\usepackage{graphicx}
\usepackage{MnSymbol}

\usepackage{amsthm}

\title[strongly bounded groups]{Strongly bounded groups of various cardinalities}

\theoremstyle{definition}\newtheorem{theorem}{Theorem}
\theoremstyle{definition}
\theoremstyle{definition}

\theoremstyle{definition}\newtheorem{bigtheorem}{Theorem}

\numberwithin{theorem}{section}
\theoremstyle{definition}\newtheorem{corollary}[theorem]{Corollary}
\theoremstyle{definition}\newtheorem{proposition}[theorem]{Proposition}
\theoremstyle{definition}\newtheorem{definition}[theorem]{Definition}
\theoremstyle{definition}
\theoremstyle{definition}
\theoremstyle{definition}
\theoremstyle{definition}
\theoremstyle{definition}\newtheorem{lemma}[theorem]{Lemma}
\theoremstyle{definition}
\theoremstyle{definition}
\theoremstyle{definition}
\theoremstyle{definition}
\theoremstyle{definition}\newtheorem{definitions}[theorem]{Definitions}

\newcommand{\R}{\mathcal{R}}
\newcommand{\G}{\mathcal{G}}

\newcommand{\CH}{\mathcal{CH}}
\newcommand{\cof}{\operatorname{cof}}

\begin{document}

\author{Samuel M. Corson}
\address{Instituto de Ciencias Matem\'aticas CSIC-UAM-UC3M-UCM, 28049 Madrid, Spain.}
\email{sammyc973@gmail.com}

\author{Saharon Shelah}
\address{Einstein Institute of Mathematics, The Hebrew University of Jerusalem, Jerusalem 91904 Israel}
\address{Department of Mathematics, Rutgers University, Piscataway, NJ 08854 USA}
\email{shelah@math.huji.ac.il}

\thanks{The first author's work was supported by the European Research Council grant PCG-336983 and by the Severo Ochoa Programme for Centres of Excellence in R\&D SEV-20150554.}
\thanks{The second author's work was supported by the European Research Council grant 338821.  Paper number 1169 on Shelah's archive.  A new 2019 version of the second author's paper number 1098 will in some respect continue this paper on other problems and cardinals.}

\begin{abstract}  Strongly bounded groups are those groups for which every action by isometries on a metric space has orbits of finite diameter.  Many groups have been shown to have this property, and all the known infinite examples so far have cardinality at least $2^{\aleph_0}$.  We produce examples of strongly bounded groups of many cardinalities, including $\aleph_1$, answering a question of Yves de Cornulier \cite{dC}.  In fact, any infinite group embeds as a subgroup of a strongly bounded group which is, at most, two cardinalities larger.
\end{abstract}

\maketitle

\begin{section}{Introduction}

In geometric group theory one extracts information regarding groups via actions on metric spaces.  Little knowledge can be gleaned from a group action which has bounded orbits, and so one often uses non-geometric approaches for the study of, say, a finite group.  Interestingly there are infinite groups which are similarly not suited for study using geometric techniques.  A group $G$ is \emph{strongly bounded} if every action of $G$ by isometries on a metric space has bounded orbits \cite{dC} (this is sometimes referred to as \emph{the Bergman property}).  We emphasize that we are considering all abstract actions of $G$ on all metric spaces, regardless of any natural topology which $G$ may carry.  Examples of infinite strongly bounded groups were produced by the second author in \cite{Sh} using extra set theoretic assumptions, and more recently Bergman showed that the full symmetric group on a set is strongly bounded \cite{B}.  The group of self-homeomorphisms of the Cantor set and of the irrational numbers \cite{DG}, $\omega_1$-existentially closed groups, and arbitrary powers of a finite perfect group are also strongly bounded \cite{dC}.

All infinite strongly bounded groups are necessarily uncountable (see \cite[Remark 2.5]{dC}), and all known infinite examples so far have cardinality at least $2^{\aleph_0}$.  It is natural to ask whether there exists a strongly bounded group of cardinality $\aleph_1$ (see \cite[Question 4.16]{dC}).  We give an affirmative answer to this and many other such questions (see Section \ref{smallcancellation} for set theoretic definitions):

\begin{bigtheorem}\label{justfromZFC}  Let $\lambda$ be a cardinal of uncountable cofinality and $K$ be a group such that $|K| < \lambda$.  Then there exists a strongly bounded group $G \geq K$ which is of cardinality $\lambda$, except possibly when $\lambda = \mu^+$ where $\cof(\mu) = \omega$ and $\mu$ is a limit of weakly inaccessible cardinals.
\end{bigtheorem}

Thus, for example, there exist strongly bounded groups of cardinality $\aleph_1$, $\aleph_2$, $\aleph_{\omega +1}$, and $\aleph_{\aleph_1}$.  Moreover, if an infinite group is of cardinality $\kappa$ then it embeds as a subgroup of a strongly bounded group of cardinality $\kappa^{++}$, though often the strongly bounded group can be made to have cardinality $\kappa^+$ instead.  The proof utilizes small cancellation over free products.  One cannot drop the assumption regarding uncountable cofinality: a group which is countably infinite, or uncountable of cardinality which is $\omega$-cofinal, cannot be strongly bounded.  It is already known that any group $K$ embeds in a strongly bounded group of cardinality $|K|^{\aleph_0}$ (see \cite[Corollary 3.2]{dC}).

By assuming some extra set theory we can produce other examples of strongly bounded groups of cardinality $\aleph_1$ which seem slightly more tame.  In the next theorem, the hypothesis $\cof(LM) = \aleph_1$ is equivalent to the assertion that there exists an increasing sequence $\{X_{\alpha}\}_{\alpha<\aleph_1}$ of sets of Lebesgue measure zero such that any set of measure zero eventually includes in the elements of the sequence.

\begin{bigtheorem}\label{existence}  Suppose that $\cof(LM) = \aleph_1$ and that $H$ is a nontrivial finite perfect group.  Then there exists a strongly bounded group of cardinality $\aleph_1$ which is a subgroup of $\prod_{\omega} H$.
\end{bigtheorem}

Such groups will be constructed by producing a special type of Boolean algebra and applying a result of de Cornulier.  The group $\prod_{\omega} H$ mentioned in Theorem \ref{existence} is itself already known to be strongly bounded \cite[Theorem 4.1]{dC}.  Assuming that ZF is consistent, one can produce models of ZFC in which $\cof(LM) = \aleph_1$ and also $2^{\aleph_0}$ is any cardinal which is not ruled out by the classical theorems of set theory \cite{JuKoz}.  Thus we obtain:

\begin{corollary}  If $\kappa$ is a cardinal of uncountable cofinality then there exists a model of ZFC in which $2^{\aleph_0} = \kappa$ and there is a strongly bounded group of cardinality $\aleph_1$ which is a subgroup of $\prod_{\omega}H$, where $H$ is any nontrivial finite perfect group. (Assuming, of course, that ZF is consistent.)
\end{corollary}

In Section \ref{smallcancellation} we prove Theorem \ref{justfromZFC} and in Section \ref{Boolean} we prove Theorem \ref{existence}.
\end{section}

\begin{section}{Proof of Theorem \ref{justfromZFC}}\label{smallcancellation}

In this section we will first quote an alternative characterization for a group to be strongly bounded and then review small cancellation over free products.  Then we review some set theory and furnish the proof of Theorem \ref{justfromZFC}.

 If $G$ is a group, $1_G$ denotes the identity element of $G$ and $Z\subseteq G$, we denote

\begin{center}  $\G(Z) = Z \cup \{1_G\} \cup \{g^{-1}\mid g\in Z\} \cup \{gh\mid g, h\in Z\}$.
\end{center}

\begin{lemma}\label{alternative}(\cite[Proposition 2.7]{dC}) A group $G$ is strongly bounded if and only if for every sequence $\{Z_m\}_{m\in \omega}$ of subsets of $G$ such that $\G(Z_m) \subseteq Z_{m+1}$ and $\bigcup_{m\in \omega} Z_m = G$ there exists an $m\in \omega$ for which $Z_m = G$.
\end{lemma}

Now for the review of free products (see \cite[V.9]{LS}).  Recall that elements of a free product $F = *_{i\in I} H_i$ are naturally viewed as words whose letters are nontrivial elements of $\bigcup_{i\in I} H_i$.  We will write $w\equiv u$ to say that two such words are equal as words, letter for letter, and write $w = u$ if the group element given by the product $w$ is equal in $F$ to the group element given by the product $u$.  Concatenation of words $w$ and $u$ will be denoted as usual by $wu$, meaning that one writes the word $w$ and then to the right of this one writes the word $u$.  

Each element $g$ of $F$ has a unique writing as such a word $g = w \equiv g_1\cdots g_k$ which is of minimal length (the \emph{normal form}) in which no two consecutive letters in the word are elements of the same $H_i$, and we let $L(g) = k$ denote the length of such an expression.  Given two normal forms $w \equiv g_1\cdots g_k$ and $u \equiv h_1\cdots h_j$ one computes the normal form of the group element $wu$ in the following way.  First we find $s\in \omega$ which is maximal such that $g_{k + 1 -r} = h_r^{-1}$ for all $1\leq r\leq s$ (we allow $s$ to be $0$).  In case $k-s\geq 1$ and $s+1\leq j$ we get $g_{k-s}\neq h_{s+1}^{-1}$.  If $g_{k-s}$ is in the same $H_i$ as $h_{s+1}$ then we let $g_{k-s}h_{s+1} = h\in H_i$ and obtain the normal form $g_1\cdots g_{k - s - 1} h h_{s+2}\cdots h_j$ for the group element $wu$.  Otherwise we get $g_1\cdots g_{k - s} h_{s + 1} \cdots h_j$ as the normal form.  We say a group element $w \in F$ has \emph{semi-reduced form} $uv$ if both $u$ and $v$ are normal forms, $w = uv$, and the number $s$ used in the computation for the normal form for $uv$ is $0$.

An element in $F$ with normal form $w \equiv g_1\cdots g_k$ is \emph{cyclically reduced} if either $L(w) \leq 1$, or $g_1$ and $g_k$ are in different $H_i$.  More generally we say $w$ is \emph{weakly cyclically reduced} if either $L(w) \leq 1$, or $g_1 \neq g_k^{-1}$.  A subset $R\subseteq F$ is \emph{symmetrized} if every $w\in R$ is weakly cyclically reduced and every weakly cyclically reduced conjugate of $w$ and of $w^{-1}$ is also in $R$.  From a set $\Gamma$ of weakly cyclically reduced elements of $F$ one obtains a symmetrized set by taking all weakly cyclically reduced conjugates of $\Gamma$ and then taking their inverses.  Given a symmetrized set $R$, a word $u$ is a \emph{piece} if there exist distinct $w_1, w_2\in R$ with semi-reduced forms $w_1 = uv_1$ and $w_2 = uv_2$.

\begin{definition}  A symmetrized set $R$ for the free product $F = *_{i\in I} H_i$ satisfies the \emph{$C'(\eta)$ condition}, where $\eta >0$,  if for each $w\in R$ we have

\begin{enumerate}

\item $L(w) >\frac{1}{\eta}$; and

\item whenever $w = uv$ is a semi-reduced form, with $u$ a piece, we have $L(u)<\eta L(w)$.

\end{enumerate}

\end{definition}

We use the following:

\begin{lemma}\label{LyndonandSchupp}(see \cite[Corollary V.9.4]{LS})  Let $F = *_{i\in I} H_i$ be a free product and $R$ a symmetrized subset of $F$ which satisfies $C'(\frac{1}{6})$.  Let $N$ be the normal closure of $R$ in $F$.  Then the natural map $F \rightarrow F/N$ embeds each factor $H_i$ of $F$.
\end{lemma}

\begin{lemma}\label{Smallcans}  For each $n \geq 1$ there is a group word $w(x_0, x_1, \ldots, x_{n-1}, y)$ such that the following holds: If $G$ is a group and $f:  (G \setminus \{1_G\})^n \rightarrow G$ then there exist group $H$ and $c\in H$ such that

\begin{enumerate}[(a)]

\item $G \leq H$;

\item $c\in H\setminus G$;

\item for all $\overline{g}\in (G\setminus \{1_G\})^n$ we have $w(\overline{g}, c) = f(\overline{g})$;

\item $H = \langle G \cup \{c\}\rangle$.
\end{enumerate}

\end{lemma}

\begin{proof}  Let $u(x_0, x_1, \ldots, x_{n-1}, y)$ be given by

$$x_0yx_1yx_2y\cdots x_{n - 2}yx_{n - 1}$$

\noindent and let $w(x_0, \ldots, x_{n - 1}, y)$ be given by 

$$y^kuy^{k - 1}uy^{k - 2}u\cdots y^3uy^2uyu$$

\noindent where $k = 32$.  Let $F$ be the free product given by  $F = \langle c \rangle * G$, where $c$ has infinite order.  Let $\Gamma_0 = \{(f(\overline{g}))^{-1}w(\overline{g}, c)\mid \overline{g}\in (G\setminus \{1_G\})^n\}$.  Notice that the elements of $\Gamma_0$ are weakly cyclically reduced unless $g_{n-1} = f(\overline{g})$, in which case we replace the word $(f(\overline{g}))^{-1}w(\overline{g}, c)$ with the weakly cyclically reduced word obtained by reducing the word $w(\overline{g}, c)(f(\overline{g}))^{-1}$.  By performing all these replacements we obtain a new set $\Gamma$.

Notice that the symmetrization $R$ of $\Gamma$ satisfies $C'(\frac{1}{6})$ over the free product $F$.  More specifically each element of $\Gamma$ is weakly cyclically reduced and of length $(2n-1)k + k + 1 = 2nk + 1$ in case $f(\overline{g}) \neq 1_G, g_{n-1}$; of length $2nk + 1 - 2 = 2nk - 1$ in case $g_{n-1} = f(\overline{g})$; or of length $2nk$ in case $f(\overline{g}) = 1_G$.  Weakly cyclically reduced conjugates of elements of $\Gamma$ will have length at least $2nk - 2$, similarly for the inverses of such elements.  It is clear that no normal form which has form 

\begin{center}
$$v_1c^{m_1}(u(\overline{g}, c))^{\pm 1}c^{m_2}(u(\overline{g}, c))^{\pm 1}c^{m_3}v_3$$
\end{center}

\noindent where $v_1, v_3\in F$ and $m_1, m_2, m_3 \in \mathbb{Z} \setminus \{0\}$ can be a piece.  Thus we can use, for example, $10n$ as a very na\"ive upper bound on the length of a piece.  For any $w\in R$ we have

\begin{center}
$$L(w) \geq 2nk - 2  = 64n - 2 > 6$$
\end{center}

\noindent as well as

\begin{center}
$$ 10n < \frac{1}{6}(64n-2) = \frac{1}{6}(2nk - 2)\leq \frac{1}{6}L(w)$$
\end{center}

\noindent and so $R$ indeed satisfies $C'(\frac{1}{6})$.

Let $N$ be the normal subgroup in $F$ generated by $R$ and by Lemma \ref{LyndonandSchupp} that the homomorphism $F \rightarrow F/N = H$ embeds each of $G$ and $\langle c\rangle$.  The claim is immediate.

\end{proof}

As is usual, we shall consider each ordinal number to be the set of ordinal numbers below itself (e.g. $0 = \emptyset$, $1 = \{0\}$, $\omega + 1 = \{0, 1, \ldots, \omega\}$) and the cardinal numbers to be the ordinals which cannot inject to a proper initial subinterval of themselves.  The notation $|Y|$ denotes the cardinality of the set $Y$.  A subset $X$ of ordinal $\alpha$ is \emph{bounded} if there is an upper bound $\beta<\alpha$ for all elements of $X$.  The \emph{cofinality} of an ordinal $\alpha$ (denoted $\cof(\alpha)$) is the least cardinality of an unbounded subset of $\alpha$.  An infinite cardinal $\lambda$ is \emph{regular} if $\cof(\lambda) = \lambda$, and is \emph{singular} otherwise.  We use $\kappa^+$ to denote the smallest cardinal which is strictly greater than $\kappa$, and similarly $\kappa^{++} = (\kappa^+)^+$.  An infinite cardinal $\lambda$ is a \emph{successor cardinal} if $\lambda = \kappa^+$ for some cardinal $\kappa$, and is a \emph{limit cardinal} otherwise.  An uncountable cardinal which is a limit regular cardinal is \emph{weakly inaccessible}.  For any infinite cardinal $\kappa$ the successor cardinal $\kappa^+$ is regular.

Next we remind the reader of some notation from Ramsey coloring theory.

\begin{definitions}If $X$ is a set and $n\in \omega$ we let $[X]^{n}$ denote the set of subsets of $X$ of cardinality $n$.  If $\kappa$, $\lambda$, and $\mu$ are cardinals and $n\in \omega$ then we write

\begin{center}  $\lambda \rightarrow [\mu]^n_{\kappa}$

\end{center}

\noindent to mean than if $f: [\lambda]^n \rightarrow \kappa$ is any function then for some $A\subseteq \lambda$ with $|A| = \mu$ we have $f([A]^n)$ is a proper subset of $\kappa$ (see \cite{EHR}).  The negation of this relation is denoted $\lambda \nrightarrow [\mu]^n_{\kappa}$.  The reader should take care not to confuse this square bracket partition relation with the parenthetical notation $\lambda \rightarrow (\mu)^n_{\kappa}$.
\end{definitions}

\begin{proof}[Proof of Theorem \ref{justfromZFC}]  The relation $\oplus_{\lambda, n}$, where $\lambda$ is an infinite cardinal and $n\in \omega$, will mean that there exists some $f: [\lambda]^n \rightarrow \lambda$ such that if $h:\lambda \rightarrow \omega$ is any function then for some $m\in \omega$ we have 

\begin{center}
$\lambda = \{f(Z)\mid Z \subseteq \{\alpha<\lambda\mid h(\alpha)<m\}\text{ and }|Z| = n\}$.
\end{center}

Clearly $\lambda \nrightarrow [\lambda]^n_{\lambda}$ implies $\oplus_{\lambda, n}$.  We note that if $\lambda$ is a successor of a regular cardinal, or if $\lambda = \mu^+$ where $\mu$ is singular and not a limit of weakly inaccessible cardinals, then $\lambda \nrightarrow [\lambda]_{\lambda}^2$, and therefore $\oplus_{\lambda, 2}$, holds (see \cite[Theorems 3.1, 3.3(3)]{Sh2} and \cite{Tod}).  We consider three cases.  

\textbf{Case I: $\oplus_{\lambda, n}$ holds, $\cof(\lambda)>\omega$.}  In this case we let $K$ be a group, without loss of generality infinite, with $|K|<\lambda$.  The construction is by induction.  First we define an increasing sequence of ordinals $\{\beta_{\alpha}\}_{\alpha<\lambda}$ by letting $\beta_0 = |K|$, $\beta_{\alpha + 1} = \beta_{\alpha} + \beta_{\alpha}$ and $\beta_{\alpha} = \bigcup_{\gamma<\alpha}\beta_{\gamma}$ when $\alpha$ is a limit ordinal.

Next we let $f: [\lambda\setminus\{0\}]^n \rightarrow \lambda\setminus \{0\}$ witness $\oplus_{\lambda, n}$.  We can without loss of generality assume that $f(W) \in \beta_{\alpha}$ for all $\alpha<\lambda$ and $W \in[\beta_{\alpha}]^n$.  To see this, let $U$ be a set such that $|U| = \lambda$ and by assumption let $g: [U]^n \rightarrow U$ be such that for any $h: U \rightarrow \omega$ there exists $m\in \omega$ for which

\begin{center}
$U = \{g(W)\mid W \subseteq \{x\in U \mid h(x)<m\}\text{ and }|W| = n\}$.
\end{center}

\noindent Pick a well order $U = \{x_{\epsilon}\}_{\epsilon<\lambda}$.  Given a subset $W \subseteq U$ we let $g''(W) = \bigcup_{k\in\omega} W_k$ where $W_0 = W$ and $W_{k+1} = W_k \cup \{g(W)\mid W\subseteq W_k \text{ and }|W| = n\}$.  Let $U_0'\subseteq U$ be such that $|U_0'| = |K|$.  Let $U_0 = g''(U_0')$.  If $\alpha<\lambda$ is a limit ordinal we let $U_{\alpha} = \bigcup_{\gamma<\alpha}U_{\gamma}$.  If $\alpha = \gamma + 1$ then we pick $U \supseteq U_{\alpha}' \supseteq U_{\gamma}$ such that the minimal element of $U \setminus U_{\gamma}$ is in $U_{\alpha}'$ and $|U_{\alpha}'| = |U_{\alpha}' \setminus U_{\gamma}| = |U_{\gamma}|$.  Let $U_{\alpha} = g''(U_{\alpha}')$.  Notice that $g([U_{\alpha}]) \subseteq U_{\alpha}$ for each $\alpha < \lambda$.  By the induction we also have $U = \bigcup_{\alpha<\lambda} U_{\lambda}$.  Taking $p:U \rightarrow \lambda\setminus \{0\}$ to be any bijection such that $p(U_{\alpha}) = \beta_{\alpha}$ for all $\alpha$ and defining $f = p \circ g\circ p^{-1}$ we obtain the required $f$.

We define the group $G$ to have set of elements $\lambda$ and give it a group structure as an increasing union of subgroups $G_{\alpha}$, with $G_{\alpha}$ having $\beta_{\alpha}$ as its underlying set of elements.  Define $G_0$ to have the group structure of $K$ on the set of elements $\beta_0$ with $0$ identified with the trivial group element $1_K$.  If we have defined the group structure $G_{\gamma}$ for all $\gamma<\alpha \leq \lambda$ and $\alpha$ is a limit ordinal then we let $G_{\alpha}$ have the unique group structure imposed by the $G_{\gamma}$ with $\gamma<\alpha$.  If $\lambda >\alpha = \gamma + 1$ then by Lemma \ref{Smallcans} we define $G_{\alpha}$ to have group structure such that

\begin{enumerate}[(a)]

\item $G_{\gamma} \leq G_{\alpha}$;

\item for all $\overline{g}\in (G_{\gamma}\setminus \{1_{G_{\gamma}}\})^n$ such that $g_0<g_1< \cdots< g_{n-1}$ we have $w(\overline{g}, \beta_{\gamma}) = f(\{g_0, \ldots, g_{n-1}\})$;

\item $G_{\alpha} = \langle G_{\gamma} \cup \{\beta_{\gamma}\}\rangle$
\end{enumerate}

\noindent (here we use the fact that $|\beta_{\alpha}| = |\beta_{\alpha}\setminus \beta_{\gamma}| = |\beta_{\gamma}|$).

Now $G = G_{\lambda}$ and we let $X = \{\beta_{\alpha}\}_{\alpha<\lambda}$.  Suppose that $\{Z_m\}_{m\in \omega}$ is a sequence of subsets of $G$ such that $G = \bigcup_{m\in \omega} Z_m$ and $Z_{m + 1} \supseteq \G(Z_m)$.  Select $m\in \omega$ large enough that

\begin{center}
$\lambda \setminus \{0\}= \{f(W)\mid W \subseteq \{0 \neq \alpha<\lambda\mid \alpha \in Z_m\}\text{ and }|W| = n\}$
\end{center}

\noindent and that $1_G \in Z_m$ and that $X \cap Z_m$ is unbounded in $\lambda$.  Given arbitrary $g\in G\setminus \{1_G\}$ we select nontrivial $g_0< \ldots< g_{n-1}$ in $Z_m$ such that $f(\{g_0, \ldots, g_{n-1}\}) = g$.  Pick $g_n\in X \cap Z_m$ which is larger than all $g_0, \ldots,  g_{n-1}$.  Then we have $w(g_0, \ldots, g_{n-1}, g_n) = f(\{g_0, \ldots, g_{n-1}\}) = g$ and so $G = Z_{m + j}$ where $j$ is the length of the word $w$.  Case I is proved.

\textbf{Case II: $\lambda$ is a limit cardinal, $\cof(\lambda)>\omega$.}  In this case we let $\lambda = \bigcup_{\alpha < \cof(\lambda)} \lambda_{\alpha}$ where $\{\lambda_{\alpha}\}_{\alpha<\cof(\lambda)}$ is a strictly increasing sequence of cardinals below $\lambda$ such that $\lambda_0 \geq |K|$.  Notice that each cardinal $\lambda_{\alpha}^{++}$ satisfies Case I.  Let $\lambda_0$ be given any group structure such that $K$ is a subgroup in $\lambda_0$.  By Case I we let $\lambda_0^{++}$ be given a group structure $G_0$ which is strongly bounded.  For each $\alpha<\cof(\lambda)$ we endow $\lambda_{\alpha}^{++}$ with a group structure $G_{\alpha}$ which extends the group structure on $\bigcup_{\gamma<\alpha}\lambda_{\gamma}^{++}$ and such that $G_{\alpha}$ is strongly bounded (by Case I).  Now let $\lambda$ be given the group structure $G$ inherited from all the $G_{\alpha}$.  Let $\{Z_m\}_{m\in \omega}$ be such that $\G(Z_m) \subseteq Z_{m+1}$ and $\bigcup_{m\in \omega}Z_m = G$ and notice that for each $\alpha<\cof(\lambda)$ there exists some minimal $m_{\alpha}\in \omega$ such that $Y_{m_{\alpha}}\cap G_{\alpha} = G_{\alpha}$.  Then $\alpha\mapsto m_{\alpha}$ is a nondecreasing sequence from $\cof(\lambda)$ to $\omega$, so it eventually stabilizes, and so $G$ is strongly bounded.

\textbf{Case III: $\lambda = \mu^+$ where $\cof(\mu) >\omega$ and $\mu$ is singular.}  Let $K$ be a group of cardinality $<\lambda$.  We let $\mu = \bigcup_{\alpha<\cof(\mu)} \mu_{\alpha}$ with $\{\mu_{\alpha}\}_{\alpha<\cof(\mu)}$ being a strictly increasing sequence of cardinals.  We have $\mu_{\alpha}^{++} \nrightarrow [\mu_{\alpha}^{++}]^2_{\mu_{\alpha}^{++}}$.  Let $f_{\alpha}: [\mu_{\alpha}^{++}]^2 \rightarrow \mu_{\alpha}^{++}$ witness this.  Let $f: [\mu]^2 \rightarrow \mu$ be defined by $f(W) = f_{\alpha}(W)$ where $\alpha<\cof(\mu)$ is minimal such that $W\in [\mu_{\alpha}^{++}]^2$.

For each ordinal $\mu\leq\gamma <\mu^+ = \lambda$ we let $j_{\gamma}: \gamma \rightarrow \mu$ be any bijection and define $h_{\gamma}: [\gamma]^2 \rightarrow \gamma$ by $j_{\gamma}^{-1}\circ f\circ j_{\gamma}$ (here $j_{\gamma}(W)$, where $W\in [\gamma]^2$, means the $2$-element set obtained by applying $j_{\gamma}$ to the elements of $W$).  Define $h: [\lambda]^{3} \rightarrow \lambda$ by $h_{\max(W)}(W \setminus \{\max(W)\})$.

We define a group structure on $\lambda$ by induction.  Let $\beta_0 = \mu$, $\beta_{\delta + 1} = \beta_{\delta} + \beta_{\delta}$ and $\beta_{\delta} = \bigcup_{\epsilon<\delta} \beta_{\epsilon}$ when $\delta$ is a limit ordinal.  Let $G_0$ be any group structure on $\beta_0$ which includes $K$ as a subgroup.  If $G_{\delta}$ has been defined for all $\epsilon<\delta\leq \lambda$ and $\delta$ is a limit ordinal then we let $G_{\delta}$ be the induced group structure on $\beta_{\delta}$.  If $\lambda > \delta = \epsilon + 1$ then we let $G_{\delta}$ be the group given by

\begin{enumerate}[(a)]

\item $G_{\epsilon} \leq G_{\delta}$;

\item for all $\overline{g}\in (G_{\epsilon}\setminus \{1_{G_{\epsilon}}\})^2$ with $j_{\beta_{\epsilon}}(g_0)< j_{\beta_{\epsilon}}(g_1)$ we have $w(\overline{g}, \beta_{\epsilon}) = h(\{g_0, g_{1}, \beta_{\epsilon}\})$;

\item $G_{\delta} = \langle G_{\epsilon} \cup \{\beta_{\epsilon}\}\rangle$

\end{enumerate}

\noindent where $w$ is as in Lemma \ref{Smallcans}.  Now we have our group structure $G$ on $\lambda$.  Let $\{Z_m\}_{m\in \omega}$ be as usual.

We claim that for each $\delta<\lambda$ there exists some $m\in \omega$ such that $G_{\delta} \subseteq Z_m$.   Fix $\delta< \lambda$ and select $m_0\in \omega$ large enough that $\beta_{\delta}\in Z_{m_0}$.  Notice that for each $\alpha < \cof(\mu)$ there is some natural number $m > m_0$ for which $|\mu_{\alpha}^{++} \cap j_{\beta_{\delta}}(Z_m \cap \beta_{\delta})| = \mu_{\alpha}^{++}$ (since $\mu_{\alpha}^{++}$ is necessarily of cofinality $>\omega$).  As $\cof(\mu)>\omega$ there must exist some $m_1 > m_0$ for which 

\begin{center}

$\{\alpha < \cof(\mu)\mid |\mu_{\alpha}^{++} \cap j_{\beta_{\delta}}(Z_{m_1} \cap \beta_{\delta})| = \mu_{\alpha}^{++}\}$

\end{center}

\noindent is unbounded in $\cof(\mu)$.  Let $g\in G_{\delta}$ be given.  Select $\alpha<\cof(\mu)$ large enough that $j_{\beta_{\delta}}(g)\in \mu_{\alpha}^{++}$ and such that $|\mu_{\alpha}^{++} \cap j_{\beta_{\delta}}(Z_{m_1} \cap \beta_{\delta})| = \mu_{\alpha}^{++}$.  Select elements $\mu_{\alpha}^+ < \zeta_0 < \zeta_1 < \mu_{\alpha}^{++}$ which are elements of $j_{\beta_{\delta}}((Z_{m_1}\setminus \{1_G\}) \cap \beta_{\delta})$ for which $f_{\alpha}(\{\zeta_0, \zeta_1\}) = j_{\beta_{\delta}}(g)$.  Then $f(\{\zeta_0, \zeta_1\}) = j_{\beta_{\delta}}(g)$, and by construction it follows that $$w(j_{\beta_{\delta}}^{-1}(\zeta_0),  j_{\beta_{\delta}}^{-1}(\zeta_1), \beta_{\delta}) = g$$ and so $G_{\delta} \subseteq Z_{m_1 + j}$ where $j$ is the length of the word $w$.

Now letting $m_{\delta}$ be minimal such that $G_{\delta} \subseteq Z_{m_{\delta}}$ we get a nondecreasing function $\delta \mapsto m_{\delta}$ from $\lambda$ to $\omega$, which must stabilize.  Thus $G$ is strongly bounded.
\end{proof}

\end{section}

\begin{section}{Proof of Theorem \ref{existence}}\label{Boolean}

We assume that the reader is familiar with the definition of a Boolean algebra (see \cite[I. 7]{J}).  We shall use notation $x\wedge y$ and $x \vee y$ for the meet and join of $x$ and $y$ in a Boolean algebra, $x^c$ for the complement of $x$, $x-y = x\wedge y^c$, and $1$ and $0$ for the top and bottom elements.  Given a subset $Z$ of a Boolean algebra $\mathcal{A}$ we let $\R(Z)$ equal the following set: $$Z \cup \{0, 1\} \cup \{x^c\mid x\in Z\}\cup\{x\vee y\mid x, y\in Z\} \cup \{x\wedge y\mid x, y\in Z\}\cup \{x - y\mid x, y\in Z\}.$$

\begin{definition}A \emph{proper $\R$-filtration} of a Boolean algebra $\mathcal{A}$ is a sequence $\{Z_n\}_{n\in \omega}$ such that $Z_n$ properly includes in $Z_{n+1}$, $\R(Z_n) \subseteq Z_{n+1}$, and $\mathcal{A} = \bigcup_{n\in \omega}Z_n$.  A proper $\R$-filtration induces a function $f: \mathcal{A} \rightarrow \omega$ by letting $f(x) = \min\{n\in \omega\mid x\in Z_n\}$.
\end{definition}

\begin{definition}(see \cite[Remark 4.5]{dC})  An infinite Boolean algebra has \emph{strong uncountable cofinality} if it has no proper $\R$-filtration.
\end{definition}

We shall be especially interested in a specific type of algebra.

\begin{definition}  An \emph{algebra on a set} $X$ is a collection $\mathcal{A}$ of subsets of $X$ for which 

\begin{itemize}
\item $X\in \mathcal{A}$;

\item $Z, Z' \in \mathcal{A}$ implies $Z\cap Z'\in \mathcal{A}$; and

\item $Z\in \mathcal{A}$ implies $X \setminus Z \in \mathcal{A}$.
\end{itemize}

\noindent  Intersection, union and set theoretic complementation answer for the meet, join and complementation which endow $\mathcal{A}$ with a natural Boolean algebra structure.
\end{definition}

Given an algebra $\mathcal{A}$ on $X$ and a function $f: X \rightarrow Y$ we shall say that $f$ is \emph{measurable} if each preimage $f^{-1}(y)$ is in $\mathcal{A}$ for each $y\in Y$.  If $Y$ is a finite group then it is easy to check that the set of measurable functions from $X$ to $Y$ forms a group under component wise multiplication: $(f_0*f_1)(x) = f_0(x)f_1(x)$.

The following was essentially proved by Yves de Cornulier in \cite{dC}:

\begin{theorem}\label{deCornulier}  Suppose $\mathcal{A}$ is an algebra of sets on a set $X$ which is of strong uncountable cofinality and $H$ is a finite perfect group.  Then the group of measurable functions from $X$ to $H$ is strongly bounded.
\end{theorem}

\begin{proof}  See proof of \cite[Thm. 4.1]{dC}.
\end{proof}

Thus to prove Theorem \ref{existence} it suffices to prove the following:

\begin{proposition}\label{stronguncountable}  If $\cof(LM) = \aleph_1$ then there exists an algebra of sets on $\omega$ of cardinality $\aleph_1$ which is of strong uncountable cofinality.
\end{proposition}

This is a slight refinement of the main result of \cite{CP} in which Cielsielski and Pawlikowski construct from the assumption $\cof(LM) = \aleph_1$ an algebra of cardinality $\aleph_1$ which is of uncountable cofinality (i.e. an algebra that is not the union of a strictly increasing $\omega$ sequence of subalgebras).  Models of ZFC + $\aleph_1< 2^{\aleph_0}$ in which such an algebra exists were first constructed by Just and Koszmider \cite{JuKoz}.  Under Martin's axiom the existence of an algebra of cardinality $\aleph_1$ of uncountable cofinality implies the continuum hypothesis \cite[Prop. 5]{Kop}.  Thus one cannot hope to prove the conclusion of Proposition \ref{stronguncountable} without extra set theoretic assumptions.

The proof of Proposition \ref{stronguncountable} will follow a slight modification to the lovely proof used in \cite{CP}.  Given a set $X$ we let $[X]^{\leq n}$ denote the set of all subsets of $X$ of cardinality at most $n$.  Consistent with \cite{CP} we let $\CH$ denote the collection of all subsets $T \subseteq \omega^\omega$ of form $T = \prod_{n\in \omega}T_n$ where $T_n \in [\omega]^{\leq n+1}$.  The cardinal $\cof(LM)$ is equal to the cardinal $$\min(\{|\mathcal{F}|  \text{ }\mid \mathcal{F} \subseteq \CH \text{ and }\bigcup\mathcal{F} = \omega^{\omega}\})$$ (see \cite{BaJud}), and for our construction we will use this latter formulation.

\begin{lemma}\label{dominatewell}  For each $T\in \CH$ these exists a strictly increasing $g\in \omega^{\omega}$ such that for every strictly increasing $f\in T$ we have $f(n) < g(n)$ for all $n\in \omega$, and whenever $g(n) \geq f(m)$ we have $g(n+1) > f(m+2)$.
\end{lemma}

\begin{proof}  Let $g(0) = \max(T_0) +1$ and generally let $g(n +1) = \max(T_{g(n) +2} \cup \{g(n)\}) + 1$.  Clearly $g(n +1) \geq g(n) +1$ for all $n\in \omega$ and so $g$ is strictly increasing and moreover $g(n) > n$.  Given a strictly increasing $f\in T$ we notice that $f(0) < g(0)$ since $f(0) \in T_0$ and $f(n+1)< f(g(n) +2)< g(n+1)$ for all $n$.  Finally, suppose that 
$m, n\in \omega$ are such that $g(n) \geq f(m)$.  Then $g(n) \geq f(m) \geq m$ and so $g(n) + 2 \geq m+2$.  Now

\begin{center}  $f(m+2) \leq f(g(n) +2) \leq \max(T_{g(n) + 2}) < g(n+1)$
\end{center}

\noindent and we are finished.
\end{proof}

The argument for the next lemma follows that of \cite[Proposition 4.4]{dC}.

\begin{lemma}\label{antichain}  If $f: \mathcal{A} \rightarrow \omega$ corresponds to a proper $\R$-filtration of Boolean algebra $\mathcal{A}$ then there exists a sequence $\{a_n\}_{n\in \omega}$ for which $a_n \wedge a_m = 0$ whenever $m\neq n$ and such that $f(a_0) < f(a_1) < \cdots$.
\end{lemma}

\begin{proof}  Let $\mathcal{L} = \{a\in \mathcal{A}\mid f(\downarrow a) \text{ is unbounded in }\omega\}$ where $\downarrow a$ denotes the set of elements in $\mathcal{A}$ below $a$.  We know that $1\in \mathcal{L}$ and if $a\in \mathcal{L}$ and $a' \leq a$ then either $a'$ or $a - a'$ is in $\mathcal{L}$.  Let $c_0 = 1$ and select $a_0\in \mathcal{A}$ such that $c_1 = c_0 - a_0 \in \mathcal{L}$.  Suppose that we have selected disjoint $a_0, \ldots, a_n \in \mathcal{A}$ as well as decreasing $c_0, \ldots, c_{n+1}\in \mathcal{L}$ with $c_m = c_{m+1}\vee a_m$ and $c_{m+1}\wedge a_m = 0$ and $f(a_0) < f(a_1) \cdots < f(a_n)$.  Select $a_{n+1}'\leq c_{n+1}$ such that $f(a_{n+1}') \geq \max(\{f(a_n), f(c_{n+1})\}) + 2$.  Notice that $f(c_{n+1}) + 2 \leq f(a_{n+1}') \leq \max(\{f(c_{n+1}), f(c_{n+1} - a_{n+1}')\}) +1$, and so $f(c_{n+1} - a_{n+1}') +1 \geq f(a_{n+1}')$ and $f(c_{n+1} - a_{n+1}')\geq f(a_n) +1$.  Thus $f(a_{n+1}'), f(c_{n+1} - a_{n+1}') > f(a_n)$.  If $c_{n+1} - a_{n+1}'\in \mathcal{L}$ then let $a_{n + 1} = a_{n+1}'$ and $c_{n+2} = c_{n+1} - a_{n+1}'$, else $a_{n+1}'\in \mathcal{L}$ and we let $c_{n + 2} = a_{n+1}'$ and $a_{n + 1} = c_{n + 1} - a_{n+1}'$.  Now it is clear that the produced sequence $\{a_n\}_{n\in \omega}$ consists of disjoint elements and $f(a_0) < f(a_1) < \cdots$.
\end{proof}

For the following, cf. \cite[Lemma 3]{CP}:

\begin{lemma}\label{ruiningsequences}  If $\cof(LM) = \aleph_1$ then for every countably infinite Boolean algebra $\mathcal{A}$ there exists a family of sequences $\{a_n^{\zeta}\}_{n\in \omega, \zeta< \aleph_1}$ in $\mathcal{A}$ such that

\begin{enumerate}  \item $a_n^{\zeta} \wedge a_m^{\zeta} = 0$ whenever $\zeta < \aleph_1$ and $n\neq m$; and

\item for every proper $\R$-filtration $f$ of $\mathcal{A}$ there exists $\zeta< \aleph_1$ for which $f(a_n^{\zeta}) > n$ for all $n\in \omega$.

\end{enumerate}
\end{lemma}

\begin{proof}  Since $\mathcal{A}$ is countably infinite, and finitely generated Boolean algebras are finite, we can write $\mathcal{A}$ as the union of a strictly increasing chain $A_0 \subsetneq A_1\subsetneq \cdots$ of finite Boolean subalgebras.  By $\cof(LM) = \aleph_1$ we select a subset $\{T_{\theta}\}_{\theta< \aleph_1} \subseteq \CH$ such that $\omega^{\omega} = \bigcup_{\theta<\aleph_1} T_{\theta}$.  For each $T_{\theta}$ select a function $g_{\theta}$ as in Lemma \ref{dominatewell}.

We notice that if $f$ is a proper $\R$-filtration of $\mathcal{A}$ there exist $\theta_0, \theta_1< \aleph_1$ such that both of the following hold for all $n\in \omega$:

\begin{enumerate}[(a)] \item $g_{\theta_0}(n)> f(b) $ for every $b\in A_n$;

\item there is an antichain $b_{n, 0}, \ldots, b_{n, 2n} \in A_{g_{\theta_1}(n +1)}$ such that $$g_{\theta_0}(g_{\theta_1}(n)) + 4(n+1) + 1 < f(b_{n, 0})<\cdots < f(b_{n, 2n}).$$

\end{enumerate}
 
\noindent To see this we select a strictly increasing $h_0 \in \omega^{\omega}$ such that $f(b) <  h_0(n)$ for each $b\in A_n$.  Since $h_0\in \omega^{\omega} = \bigcup_{\theta<\aleph_1} T_{\theta}$, we select $\theta_0$ such that $h_0\in T_{\theta_0}$.  Then $g_{\theta_0}$ satisfies property (a) since $g_{\theta_0}(n) > h_0(n)$ for all $n\in \omega$.  Select a sequence $\{a_n\}_{n\in \omega}$ as in Lemma \ref{antichain}.  Define $h_1\in \omega^{\omega}$ by first letting $h_1(0) = 0$ and then selecting $a_{0, 0}$ in $\{a_n\}_{n\in \omega}$ for which $g_{\theta_0}(h_1(0)) + 5 < f(a_{0, 0})$.  Suppose we have already selected $h_1(0), \ldots, h_1(m)$ and $\{a_{r, i}\}_{0\leq r \leq m, 0\leq i\leq 2r}$ such that for each $0\leq r<m$ we have $$g_{\theta_0}(h_1(r)) + 4(r+1) +1< f(a_{r, 0}) < f(a_{r, 1}) < \cdots < f(a_{r, 2r})$$ and $a_{r, 0}, a_{r, 1}, \ldots, a_{r, 2r}\in A_{h_1(r+1)}$, and so that $$g_{\theta_0}(h_1(m)) + 4(m+1) + 1 < f(a_{m, 0}) < f(a_{m, 1}) < \cdots < f(a_{m, 2m}).$$  Select $h_1(m+1)$ such that $a_{m, 0}, \ldots, a_{m, 2m} \in A_{h_1(m+1)}$ and select further elements $a_{m+1, 0}, \ldots, a_{m+1, 2(m+1)}$ among $\{a_n\}_{n\in \omega}$ so that $$g_{\theta_0}(h_1(m+1)) + 4(m+2) + 1 < f(a_{m+1, 0})< \ldots < f(a_{m+1, 2(m+1)}).$$  Such an $h_1$ is obviously strictly increasing and $h_1\in T_{\theta_1}$ for some $\theta_1< \aleph_1$.

We know that for each $n\in \omega$ there exists a maximal $m_n\in \omega$ such that $h_1(m_n) \leq g_{\theta_1}(n)$ and certainly $m_n\geq n$; moreover by Lemma \ref{dominatewell} we know $m_{n+1} \geq m_n +2$.  We select the antichains $b_{n, 0}, \ldots, b_{n, 2(n+1)}$ for each $n\in \omega$ by letting $b_{n, i} = a_{m_n +1, i}$.  Thus (a) and (b) are both satisfied.

Formally setting $A_{g_{\theta_1}(-1)} = \emptyset$ we know that the proper $\R$-filtration $f$ is an element of the set

\begin{center}  $X_{\theta_0, \theta_1} = \prod_{n\in \omega} (g_{\theta_0}(g_{\theta_1}(n))^{A_{g_{\theta_1}(n)} \setminus A_{g_{\theta_1}(n-1)}}$.

\end{center}

\noindent Each set $\omega^{A_{g_{\theta_1}(n)} \setminus A_{g_{\theta_1}(n-1)}}$ is countably infinite and can therefore be bijected with $\omega$.  This bijection extends to a bijection of $\omega^{\omega}$ with $\prod_{n\in \omega}\omega^{A_{g_{\theta_1}(n)} \setminus A_{g_{\theta_1}(n-1)}}$.  Since $\cof(LM) = \aleph_1$ there exists a covering of cardinality $\aleph_1$ of $\prod_{n\in \omega}\omega^{A_{g_{\theta_1}(n)} \setminus A_{g_{\theta_1}(n-1)}}$ by sets of form $S^{\eta} = \prod_{n\in \omega}S_n^{\eta}$ where $S_n^{\eta}\in [\omega^{A_{g_{\theta_1}(n)} \setminus A_{g_{\theta_1}(n-1)}}]^{\leq n+1}$.  Let $J_{\theta_0, \theta_1, \eta}$ denote the set of all proper $\R$-filtrations of $\mathcal{A}$ which are in $S^{\eta} \cap X_{\theta_0, \theta_1}$ and satisfy (a) and (b) for the parameters $\theta_0, \theta_1$. 

There are only $\aleph_1$-many choices for $\theta_0, \theta_1$ and each $X_{\theta_0, \theta_1}$ can be covered by $\aleph_1$-many sets $S^{\eta}$.  Therefore it is now sufficient to construct a sequence $\{a_n\}_{n\in \omega}$ for which 

\begin{enumerate}

\item $a_n\wedge a_m = 0$ whenever $n\neq m$; and

\item $f(a_n)>n$ for any $f\in J_{\theta_0, \theta_1, \eta}$.

\end{enumerate}

Let $\{f_i:A_{g_{\theta_1}(n)} \setminus A_{g_{\theta_1}(n-1)} \rightarrow g_{\theta_1}(g_{\theta_0}(n))\}_{0\leq i\leq n}$ be the set of restrictions $f\upharpoonright A_{g_{\theta_1}(n)} \setminus A_{g_{\theta_1}(n-1)}$ where $f\in J_{\theta_0, \theta_1, \eta}$.  We inductively define a set of elements $\{d_n^i\}_{0 \leq i \leq n} \subseteq A_{g_{\theta_1}(n)}$ such that for all $j\leq i$ we have $f_j(d_n^i) > g_{\theta_0}(g_{\theta_1}(n - 1)) + 4n - 2i$.

For $i = 0$ we can select by (b) an element $d_n^0\in A_{g_{\theta_1}(n)}$ for which $f_0(d_n^0) > g_{\theta_0}(g_{\theta_1}(n - 1)) + 4n +1$.  Suppose that we have selected $d_n^i \in A_{g_{\theta_1}(n)}$, where $i<n$, such that for all $0 \leq j\leq i$ we have $f_j(d_n^i) > g_{\theta_0}(g_{\theta_1}(n)) + 4n - 2i$.  If also $$f_{i+1}(d_n^i) > g_{\theta_0}(g_{\theta_1}(n - 1)) + 4n - 2(i + 1)$$ then we set $d_n^{i+1} = d_n^i$.  Else we have $$f_{i+1}(d_n^i) \leq  g_{\theta_0}(g_{\theta_1}(n - 1)) + 4n - 2(i + 1).$$  By (b) we select an antichain $b_{n - 1, 0}, \ldots, b_{n - 1, 2n} \in A_{g_{\theta_1}(n)}$ such that $$g_{\theta_0}(g_{\theta_1}(n - 1)) + 4n + 1 < f_{i+1}(b_{n - 1, 0})<\cdots < f_{i+1}(b_{n - 1, 2n}).$$  Notice that

\begin{center}
$\max(\{f_{i+1}(d_n^i\wedge b_{n-1, k}), f_{i+1}((d_n^i)^c\wedge b_{n-1, k})\}) \geq f_{i+1}(b_{n - 1, k}) - 1$
\end{center}

\noindent since otherwise $f_{i+1}(b_{n - 1, k}) = f_{i+1}((d_n^i\wedge b_{n-1, k}) \vee ((d_n^i)^c\wedge b_{n-1, k})) \leq f_{i+1}(b_{n - 1, k}) - 1$.  Thus we may select $d\in \{d_n^i, (d_n^i)^c\}$ for which $f_{i+1}(d \wedge b_{n-1, k}) \geq  f_{i+1}(b_{n - 1, k}) - 1$ for $n +1$ elements of $\{0, \ldots, 2n\}$ by the pigeon hole principle.  Let $K \subseteq \{0, \ldots, 2n\}$ denote the set of all $k$ for which $f_{i+1}(d \wedge b_{n-1, k}) \geq  f_{i+1}(b_{n - 1, k}) - 1$.  Letting $e_k = d \wedge b_{n-1, k}$ for each $k\in K$ we have $f_{i+1}(e_k)\geq f_{i+1}(b_{n - 1, k}) - 1 > g_{\theta_0}(g_{\theta_1}(n-1)) +4n$.  This means that for all $k\in K$ we have $$f_{i + 1}(d^c \vee e_k) \geq g_{\theta_0}(g_{\theta_1}(n - 1)) + 4n - 2(i + 1) +1$$ for otherwise we would have 

\begin{center}
$g_{\theta_0}(g_{\theta_1}(n-1)) + 4n < f_{i + 1}(e_k)$

$= f_{i+1}((d^c\vee e_k)\wedge d)$

$\leq g_{\theta_0}(g_{\theta_1}(n-1)) + 4n - 2(i+1) +2$

$\leq g_{\theta_0}(g_{\theta_1}(n-1)) + 4n$
\end{center}

\noindent Next we notice that for every $0\leq j\leq i$ there is at most one $k\in K$ for which $$f_j(d^c \vee e_{k}) \leq g_{\theta_0}(g_{\theta_1}(n -1)) + 4n - 2(i +1)$$ for if distinct $k, k' \in K$ satisfied this inequality we would have

\begin{center}  $g_{\theta_0}(g_{\theta_1}(n-1)) + 4n - 2(i+1) +1 \geq f_j((d^c \vee e_k) \wedge (d^c \vee e_{k'}))$

$=f_j(d^c)$

$> g_{\theta_0}(g_{\theta_1}(n-1)) + 4n -2i - 1$

\end{center}

\noindent which is absurd.  Thus by the pigeon hole principle, since $i<n$, there exists some $k\in K$ for which $f_j(d^c \vee e_k)> g_{\theta_0}(g_{\theta_1}(n-1)) + 4n- 2(i+1)$ for all $0\leq j\leq i+1$ and we let $d_n^{i+1} = d^c \vee e_k$.  The construction of the $d_n^i$ is now complete.

Letting $\{d_n\}_{n\geq 1}$ be given by $d_n = d_n^n$ we notice that for every $f\in J_{\theta_0, \theta_1, \eta}$ we have $f(d_n) = f(d_n^n) > g_{\theta_0}(g_{\theta_1}(n-1)) + 2n$ for each $n\geq 1$.  Thus letting $c_n = d_{n+1}$ we get for all $f\in J_{\theta_0, \theta_1, \eta}$

\begin{center}
$g_{\theta_0}(g_{\theta_1}(n+1))  > f(c_n) > g_{\theta_0}(g_{\theta_1}(n)) + 2(n +1)$
\end{center}

\noindent and

\begin{center}
$g_{\theta_0}(g_{\theta_1}(n+1))  \geq f(c_n^c) \geq g_{\theta_0}(g_{\theta_1}(n)) + 2(n+1)$
\end{center}

\noindent since $c_n\in A_{g_{\theta_1}(n+1)}$.

For each $n\in \omega$ let $c_n^0 = c_n$ and $c_n^1 = c_n^c$.  We define a sequence $n_0 < n_1< \cdots$ of natural numbers, a sequence $\sigma$ of $0$s and $1$s, as well as a sequence of subsets $\omega \supseteq Z_0 \supseteq Z_1\supseteq$.  Let $n_0 = 0$.  Notice that it is either the case that there are infinitely many $k$ for which

\begin{center}
$f(c_{n_0}\wedge c_k) \geq  f(c_k) - 1$ for all $f\in J_{\theta_0, \theta_1, \eta}$ 
\end{center}

\noindent or infinitely many $k$ for which 

\begin{center}
$f(c_{n_0}^c\wedge c_k) \geq f(c_k) - 1$ for all $f\in J_{\theta_0, \theta_1, \eta}.$
\end{center}

Select $\sigma(0)$ so that for infinitely many $k\in \omega$ we have $f(c_{n_0}^{\sigma(0)}\wedge c_k) \geq f(c_k) - 1$ for $f\in J_{\theta_0, \theta_1, \eta}$  and let

\begin{center}
$Z_0 = \{k> n_0\mid f(c_{n_0}^{\sigma(0)}\wedge c_k) \geq f(c_k) - 1\text { for all  }f\in J_{\theta_0, \theta_1, \eta}\}$
\end{center}

Let $n_1= \min(Z_0)$.  Select $\sigma(1)$ so that the set

\begin{center}
$Z_1 = \{k> n_1, k\in Z_0 \mid f(c_{n_0}^{\sigma(0)} \wedge c_{n_1}^{\sigma(1)}\wedge c_k) \geq f(c_k) - 2\text { for all } f\in J_{\theta_0, \theta_1, \eta}\}$
\end{center}

\noindent is infinite.  Continuing in this manner we construct a sequence $l_m = c_{n_0}^{\sigma(0)}\wedge \cdots \wedge c_{n_m}^{\sigma(m)}$ in $\mathcal{A}$ such that $f(l_m) \geq f(c_{n_m}) - m - 1$, $l_m \geq l_{m+1}$ and $l_m \in A_{g_{\theta_1}(n_m+1)}$.  Since  

\begin{center}
$f(l_{m}) \geq f(c_{n_m}) - m - 1$

$\geq f(c_{n_m}) - n_m -1$ 

$> g_{\theta_0}(g_{\theta_1}(n_m))$

$> f(l_{m-1})$
\end{center}

\noindent for all $f\in J_{\theta_0, \theta_1, \eta}$ we get that 

\begin{center}
$f(l_m - l_{m+1}) \geq f(l_{m+1}) - 1$

$> f(c_{n_{m+1}}) - m - 2$

$>  g_{\theta_0}(g_{\theta_1}(n_{m+1})) + 2(n_{m+1} +1) - m - 2$

$\geq g_{\theta_0}(g_{\theta_1}(n_{m+1})) + 2(m + 2) - m - 2$

$> m.$

\end{center}

\noindent Thus letting $a_m = l_m - l_{m+1}$ we are done.

\end{proof}

The construction for Proposition \ref{stronguncountable} now follows that used for \cite[Theorem 1]{CP} with almost no alteration.  For completeness we provide the construction and proof below.

\begin{proof}[Proof of Proposition \ref{stronguncountable}]  As $\cof(LM) = \aleph_1$ we have by Lemma \ref{dominatewell} a set $\{g_{\theta}\}_{\theta<\aleph_1}$ of strictly increasing functions $g_{\theta}: \omega \rightarrow \omega$ such that for each $f: \omega \rightarrow \omega$ there is some $\theta<\aleph_1$ for which $f(n)< g_{\theta}(n)$ for all $n\in \omega$.  Let $\{X_m\}_{m\in \omega}$ be a partition of $\omega$ into infinite pairwise disjoint sets.  For each $\theta<\aleph_1$ we let $g_{\theta}'(m) = \min(X_m \cap (g_{\theta}(m), \infty))$.  Given any sequence $\overline{a} = \{a_n\}_{n\in \omega}$ of pairwise disjoint subsets of $\omega$ we let $$(\overline{a})^m = \bigcup_{n\in X_m} a_n$$ and $$(\overline{a})^{\theta} = \bigcup_{m\in \omega} a_{g_{\theta}'(m)}.$$  Moreover we let $$F(\overline{a}) = \{(\overline{a})^m\mid m\in \omega\}\cup \{(\overline{a})^{\theta}\mid \theta<\aleph_1\}.$$

The Boolean algebra $\mathcal{A}$ will be constructed by induction over the ordinals less than $\aleph_1$.  Let $\mathcal{A}_0$ be a Boolean algebra on $\omega$ of cardinality $\aleph_1$.  Whenever $\epsilon<\aleph_1$ is a limit ordinal we let $\mathcal{A}_{\epsilon} = \bigcup_{\delta<\epsilon}\mathcal{A}_{\delta}$.  Construct $\mathcal{A}_{\delta + 1}$ from $\mathcal{A}_{\delta}$ by letting $\mathcal{A}_{\delta} = \{b_{\gamma}\}_{\gamma<\aleph_1}$ be an enumeration and for each $\omega \leq \alpha<\aleph_1$ let $\mathcal{A}_{\delta, \alpha}$ be the Boolean subalgebra generated by $\{b_{\gamma}\}_{\gamma< \alpha}$.  Since $\mathcal{A}_{\delta, \alpha}$ is countably infinite we select, by Lemma \ref{ruiningsequences}, sequences $\{\overline{a}^{\zeta, \alpha}\}_{\zeta<\aleph_1}$ such that $a_n^{\zeta, \alpha}\wedge a_m^{\zeta, \alpha} = 0$ when $m\neq n$ and for any proper $\R$-filtration $f$ of $\mathcal{A}_{\delta, \alpha}$ there exists $\zeta<\aleph_1$ for which $f(a_n^{\zeta, \alpha}) > n$ for all $n\in \omega$.  Let $\mathcal{A}_{\delta + 1}$ be the Boolean algebra generated by $\mathcal{A}_{\delta}\cup \bigcup_{\omega \leq \alpha<\aleph_1, \zeta<\aleph_1} F(\overline{a}^{\zeta, \alpha})$.  Let $\mathcal{A} = \bigcup_{\delta<\aleph_1} \mathcal{A}_{\delta}$.

We check that $\mathcal{A}$ is as required.  Certainly the cardinality of $\mathcal{A}$ is correct.  To see that $\mathcal{A}$ is of strong uncountable cofinality we suppose for contradiction that $f: \mathcal{A} \rightarrow \omega$ is a proper $\R$-filtration.  Select elements $b_n\in \mathcal{A}$ such that $f(b_n) > n$.  Then $\{b_n\}_{n\in \omega} \subseteq \mathcal{A}_{\delta}$ for some $\delta<\aleph_1$, and therefore $\{b_n\}_{n\in \omega} \subseteq \mathcal{A}_{\delta, \alpha}$ for some $\alpha<\aleph_1$.  The restriction $f\upharpoonright \mathcal{A}_{\delta, \alpha}$ is therefore a proper $\R$-filtration of $\mathcal{A}_{\delta, \alpha}$.  Letting $\{\overline{a}^{\zeta, \alpha}\}_{\zeta<\aleph_1}$ be the sequence selected for $\mathcal{A}_{\delta, \alpha}$, we know for some $\zeta<\aleph_1$ that $f(a_n^{\zeta, \alpha}) > n$ for all $n\in \omega$.

Now $F(\overline{a}^{\zeta, \alpha}) \subseteq \mathcal{A}_{\delta + 1} \subseteq \mathcal{A}$.  Select $\theta <\aleph_1$ for which $f((\overline{a}^{\zeta, \alpha})^m) + m + 1 < g_{\theta}(m)$ for all $m\in \omega$.  Now $f((\overline{a}^{\zeta, \alpha})^{\theta}) = m$ for some $m\in \omega$.  We notice that $(\overline{a}^{\zeta, \alpha})^{\theta} \cap (\overline{a}^{\zeta, \alpha})^m = a_{g_{\theta}'(m)}$, whence

\begin{center}
$g_{\theta}'(m) < f(a_{g_{\theta}'(m)})$

$\leq \max(\{m, f((\overline{a}^{\zeta, \alpha})^m)\}) + 1$

$< g_{\theta}(m)$

$< g_{\theta}'(m)$
\end{center}

\noindent which is a contradiction.
\end{proof}

\end{section}

\end{document}